\documentclass[11pt]{article}
\usepackage[english]{babel}
\usepackage[margin=1in]{geometry}
\usepackage[T1]{fontenc}
\usepackage[utf8]{inputenc}
\usepackage{amsmath,amsbsy,amssymb,amsthm,amsfonts,mathtools}
\usepackage{graphicx} 
\usepackage[hidelinks]{hyperref}
\mathtoolsset{showonlyrefs,showmanualtags}

\usepackage{bm}
\usepackage{bbm}
\usepackage{cancel}
\usepackage{array}
\usepackage{dsfont}
\usepackage{bbold}	
\usepackage{theoremref} 
\usepackage{enumitem}
\usepackage[normalem]{ulem}
\usepackage{xcolor}
\usepackage{commath}
\usepackage{float}
\usepackage{hyperref}

\usepackage[numbers]{natbib} 
\usepackage{color}

\numberwithin{equation}{section}
\allowdisplaybreaks

\newcommand{\R}{\mathbbm{R}}
\newcommand{\Q}{\mathbbm{Q}}

\newcommand{\1}{\mathbbm{1}}
\newcommand{\N}{\mathbbm{N}}
\newcommand{\M}{\mathbbm{M}}

\let\phi\varphi

\newcommand{\sB}{\mathcal{B}}

\newcommand{\sA}{\mathcal{A}}

\newcommand{\sP}{\mathcal{P}}
\newcommand{\sL}{\mathcal{L}}

\newcommand{\sU}{\mathcal{U}}
\newcommand{\sX}{\mathcal{X}}

\newcommand{\sS}{\mathcal{S}}

\newcommand{\sT}{\mathcal{T}}

\newcommand{\F}{\mathbb{F}}
\newcommand{\E}{\mathbb{E}}
\newcommand{\sF}{\mathcal{F}}

\newcommand{\Break}{\vspace{5mm}}

\DeclareMathOperator*{\argmax}{arg\,max}

\usepackage{amssymb}

\newtheorem{theorem}{Theorem}[section]
\newtheorem{definition}[theorem]{Definition}

\newtheorem{assumption}[theorem]{Assumption}
\newtheorem{lemma}[theorem]{Lemma}

\newtheorem{remark}[theorem]{Remark}

\newtheorem{proposition}[theorem]{Proposition}

\usepackage{fancyhdr}
\begin{document}
%Title Contents
\title{%\vspace{-1.25cm} 
Mean Field Games of Control and Cryptocurrency Mining
}

\author{
Nicol\'{a}s Garc\'{i}a\footnote{Department of Operations Research and Financial Engineering, Princeton University, Princeton, NJ, 08540, USA.}
\footnote{Acknowledges support from NSERC CGSD.}
\and
Ronnie Sircar\footnotemark[1]
\and
H. Mete Soner\footnotemark[1]\hspace{0.5em}\footnote{Partially supported by the National Science Foundation grant DMS 2406762.}
}

%Title Page
\maketitle

\begin{abstract}
This paper studies Mean Field Games (MFGs) in which agent dynamics are given by jump processes of controlled intensity, with mean-field interaction via the controls and affecting the jump intensities. We establish the existence of MFG equilibria in a general discrete-time setting, and prove a limit theorem as the time discretization goes to zero, establishing equilibria in the continuous-time setting for a class of MFGs of intensity control. This motivates numerical schemes that involve directly solving discrete-time games as opposed to coupled Hamilton-Jacobi-Bellman and Kolmogorov equations. As an example of the general theory, we consider cryptocurrency mining competition, modeled as an MFG both in continuous and discrete time, and illustrate the effectiveness of the discrete-time algorithm to solve it.  
\end{abstract}
%Table of contents page
\renewcommand{\baselinestretch}{0.5}\normalsize
\tableofcontents
\renewcommand{\baselinestretch}{1.0}\normalsize

\section{Introduction}% and Review of the Literature}

The study of Mean Field Games (MFGs) dates back to the foundational papers \cite{huang2006large} and \cite{lasry2007mean}, and we refer to \cite{carmona2018probabilistic} for a comprehensive review of the literature, most of which has focused on continuous-time models. There are comparatively few works treating discrete-time setups, of which we highlight \cite{gomes2010discrete}, which considers a finite state and horizon setup where agents control their transition probabilities, \cite{biswas2015mean} which allows for a Polish state space incorporating mean-field interactions only through the costs, and a discrete-time MFG with countable state space studied in \cite{adlakha2015equilibria}. Additionally, \cite{ni2015discrete}, \cite{moon2015discrete}, and \cite{nourian2013linear} consider linear dynamics in a variety of discrete-time settings, and \cite{saldi2018markov} analyzes the infinite-horizon discounted-cost problem with a Polish state space and MFG interaction via the states, among others.    

Motivated by the cryptocurrency mining MFG model in \cite{li2019mean}, we study a class of problems set both in discrete and continuous time in which the controlled dynamics follow a jump Markov process and the mean-field interaction is via the controls. The latter property complicates the continuous-time analysis due to the loss of regularity in the control  measure flow when compared to the state measure flow in the more typical setup involving mean-field interactions of state. MFGs with jump process dynamics in the continuous-time and finite state setting are considered in \cite{bayraktar2018analysis,gomes2013continuous}, where agents' controls are their transition probabilities and they interact via the empirical (joint) state measures. The techniques involved are primarily based on ODE and PDE methods related to the associated master equations. A probabilistic treatment of MFGs with jumps is given in \cite{benazzoli2020mean}, where relaxed controls are used in a weak formulation to provide general existence results for MFGs with jump-diffusion dynamics where both jump size and intensity are controlled, and where the mean-field interaction is limited to the states. In our analysis, we also utilize relaxed control techniques as developed in  \cite{lacker2015mean} for diffusion MFGs of state and \cite{djete2023mean} for diffusion MFGs of control. Our results complement \cite{bertucci2020mean,taghizadeh2020mean}, and directly apply to the cryptocurrency MFG model from \cite{li2019mean} for which we provide theoretical existence guarantees.  

We begin our analysis by considering discrete-time finite-horizon MFGs in Section \ref{sec:discrete}, in which we let the state, noise, and control spaces be arbitrary Polish, and consider general transition dynamics. Using the methodology from \cite{saldi2018markov} which in turn was inspired by \cite{jovanovic1988anonymous}, we characterize the MFG  equilibria as fixed points of a set-valued operator and establish fixed-point existence by way of Kakutani's theorem.   In comparison to \cite{saldi2018markov}, we allow for mean-field interaction via the controls, and prove existence in the finite-time horizon setting under a weaker growth assumption on the transition dynamics. 

In Section \ref{sec:cts}, we consider a concrete continuous-time MFG with state dynamics given by a jump process and prove that it arises as the limit of analogous discrete-time models.  We assume that the drift and the intensity coefficients do not depend on the state process, and the MFG interaction is via the controls affecting only the intensity but not the drift of the agent's jump processes. We consider this concrete setup in order to avoid routine but technical details which can be found in the classical approximations literature for stochastic optimal control including \cite{ethier2009markov,kushner2012weak,kushner2001numerical}.

In Section \ref{sec:crypto}, we apply our results from Sections \ref{sec:discrete} and \ref{sec:cts} to establish MFG existence for the cryptocurrency mining MFG model of \cite{li2019mean} as well as its discrete-time analogue.  We compute the discrete-time MFG using damped fixed point iterations, and reproduce the qualitative equilibrium behavior which was established  in \cite{li2019mean} using a  finite difference scheme for the associated coupled PDEs. We remark on uniqueness of equilibrium in Section \ref{sec:uniq}.

\textbf{Notation:} For an integer $d$, let $\sS^d$ denote the $d$-dimensional simplex, $[d] \coloneqq \{1,2,\cdots,d\}$, and define $\Delta_i : \R^d \to \R^d$ by $\Delta_i x \coloneqq (x_1 - x_i,x_2 - x_i,\cdots,x_d - x_i)$. Given a Polish space $\sX$, let $\sB(\sX)$ denote its Borel sets and let $\sP(\sX)$ denote the set of probability measures on $\sB(\sX)$.  We endow this space with the weak topology, i.e.,
$\mu_n \to \mu$ if $\lim_{n \to \infty} \int_{\sX} f d\mu_n = \int_{\sX} f d\mu$ for any continuous and bounded $f:\sX \to \R$. To emphasize convergence in the weak sense, we write $\mu_n \xrightarrow{\sL,n \to \infty} \mu$.
Let $\R_{\geq 0}$ denote the non-negative reals. For a random variable $X$, $\sL(X)$ denotes its distribution. For $\mu\in \sP(\R)$, we write $\overline{\mu} \coloneqq \int_{\sX}x\mu(dx)$, and $\delta_{\{a\}}$  denotes the Dirac measure located at $a$.

\section{Mean Field Games in Discrete Time}
\label{sec:discrete}
This section establishes the existence of  equilibria  for discrete-time finite-horizon MFGs where the mean-field interaction, which affects the dynamics and costs, is through both the controls and states. 
We motivate the problem by first introducing the $N$-player game.

\subsection{$N$-Player Game Formulation}
We begin with the following $N$-player setup where:
\begin{itemize}[itemsep=0.5pt]%[nolistsep]
\item The agents' state processes take values in a Polish space $\sX$ and evolve in discrete time steps $t = 0,1,2,\cdots,T$. We denote by $x_t^i$ the state of agent $i \in [N] \coloneqq \{1,2,\cdots,N \}$ at time $t$.
\item The initial states $x_0^i$ are i.i.d. drawn from a measure $\mu_0 \in \sP(\sX)$.
\item The action space $\sA$ is assumed to be Polish, $a_t^i \in \sA$ denotes the action of agent $i \in [N]$ at time $t$, and $e_t \coloneqq (e_t^c,e_t^s) \in \sP(\sA) \times \sP(\sX)$ denotes the empirical distribution of the agents' controls and states at time $t$.
\item The transition dynamics of each agent are given by a Markov transition kernel 
\begin{align*}
\rho : \sX \times \sA \times \sP(\sA) \times \sP(\sX) \to \sP(\sX) \quad \text{ so that } \quad
x_{t+1}^i \sim \rho(\cdot | x_t^i,a^i_t,e_t^c,e_t^s)
\end{align*}
for each $t = 0,1,2,\cdots,T-1$.
\end{itemize}
A control policy for player $i$ consists of a sequence $(\pi_t^i)_{t=0}^{T-1}$ of $\sP(\sA)$-valued random variables adapted to the filtration
\begin{align*}
\sF_0^i \coloneqq \sigma(x_0^i,e_0^s), \quad \sF_t^{i} \coloneqq \sigma\Big( \sF_{t-1} \cup \sigma (x_t^i,a_{t-1}^i,e_{t-1}^c,e_{t}^s ) \Big), \quad \text{ for all } t = 0,1,\cdots,T-1.
\end{align*}
Conditioned on $\sF_t$, the action $a_t^i$ of agent $i$ is drawn randomly (and independent of any other random quantity) from the distribution $\pi_t^i$ (i.e. $a_t^i \sim \pi_t^i$). 

The last step in the specification of the model is to define the optimality criterion, which is in terms of the  one-step running and terminal cost functions
\begin{align*}
c:\sX \times \sA \times \sP(\sA) \times \sP(\sX) \to [0,\infty), \qquad \phi:\sX \times \sP(\sA) \times \sP(\sX) \to [0,\infty).
\end{align*}
Fixing an $N$-tuple of control policies $\pi^{(N)} \coloneqq (\pi^{N,1},\pi^{N,2},\cdots,\pi^{N,N})$, the $i$-th agent incurs a cost
\begin{align*}
J^{i}(\pi^{(N)}) \coloneqq J^{i}(\pi^{N,i},\pi^{N,-i}) \coloneqq \E^{\pi^{(N)}} \Big[ \sum_{t=0}^{T-1} 
 c(x_t^i,a_t^i,e_t) + \phi(x^i_T,e_T) \Big],
\end{align*}
where the superscript $\pi^{(N)}$ denotes that the control actions of the agents are determined according to their respective policies from $\pi^{(N)}$. Agents wish to select their respective policies to minimize costs, 
and a solution to the $N$-player game consists of a  equilibrium, which is a joint policy $\tilde{\pi}^{(N)}$ such that
%\begin{align}
$$J^{i}(\tilde{\pi}^{(N)}) = \inf_{\pi^i} J^{i}(\pi^{i},\tilde{\pi}^{N,-i}),$$ 
%\end{align}
for every $i = 1,\cdots,N$. This completes the $N$-player setup.

\subsection{Mean Field Game Formulation}

We proceed with the reference agent problem which characterizes the corresponding MFG. We denote the state and control of the reference agent at time $t$ by $x_t\in\sX$ and $a_t\in\sA$, respectively. In this case, the reference agent has identical dynamics as in the $N$-player game, but the sequences of empirical measures $(e_t^s)_{t=0}^{T}$ and $(e_t^c)_{t=0}^{T}$ are taken to be deterministic. Denoting $\delta_t = (e_t^c,e_t^s) \in \sP(\sA \times \sX)$, 
the dynamics follow
\begin{align} \label{ref_ag_discrete_dyn}
x_{t+1} \sim \rho(\cdot|x_t,a_t,\delta_t) , \quad a_t \sim \pi_t(\cdot|x_t), \quad t = 0,1,\cdots,T,
\end{align}
where, because $\delta \coloneqq (\delta_t)_{t=0}^{T}$ is now a parameter, control policies become random measures adapted to the filtration of the state process. We let $\Pi$ denote the set of such policies, but will in fact search for an optimal policy within the following smaller set of Markov control policies:

\begin{definition} \label{markovdef}
A control policy is called Markov if it is a sequence $(\pi_t)_{t=0}^{T-1}$ where each $\pi_t$ is measurable w.r.t. $\sigma(x_t)$. Under such a policy, the measure $\pi_t(\cdot|x_t)$ used to generate the control at a given time $t$ depends only on the state at that time. We denote by $\M$ the set of all such policies. 
\end{definition}
It is well known that in a Markov Decision Problem (MDP) setting, the restriction to Markov policies does not result in a larger value function (see for example \cite[Proposition 3.2]{saldi2018markov}). Given a fixed sequence of probability measures $\delta\coloneqq (\delta_t)_{t=0}^{T} \subseteq \sP(\sA) \times \sP(\sX)$, the reference agent's control problem consists of determining a Markov policy $\pi^*$ such that 
\begin{align} \label{somecosteq}
J(\pi^*,\delta) = \inf_{\pi \in \M} J(\pi,\delta) \quad \text{ for } \quad J(\pi,\delta) \coloneqq  E^\pi\Big[ \sum_{t=0}^{T-1} c(x_t,a_t,\delta_t) + \phi(x_T,\delta_T) \Big],
\end{align}
with dynamics given by \eqref{ref_ag_discrete_dyn}. 

For convenience in defining an MFG equilibrium, we introduce
\begin{align*}
\Phi: (\sP(\sA) \times \sP(\sX))^{T+1} \to 2^\Pi \quad \text{ by } \quad  \Phi(\delta) \coloneqq \{\pi^* \in \Pi : \pi^* \text{ minimizes}\ J(\cdot,\delta) \}.
\end{align*}
We additionally define a map $\Lambda : \Pi \to (\sP(\sA) \times \sP(\sX))^{T+1}$ by constructing $\Lambda(\pi) \coloneqq (\delta_t^c,\delta_t^s)_{t=0}^{T} $ iteratively using the initial state law $\mu_0$ as follows: First, set $\delta_0^s = \mu_0$.

For all other $t \geq 0$, define
\begin{align*}
\delta_t^c(\cdot) = \int_{\sX} \sP_t^\pi(\cdot|x) \delta_t^s(dx), \quad \delta_{t+1}^s(\cdot) = \int_{\sX} \int_{\sA} \rho(\cdot|x,a,\delta_t)\sP_t^\pi(da|x)\delta_t^s(dx),
\end{align*}
where $\sP_t^\pi(\cdot | x)$ denotes the conditional law of $a_t$ given the event $\{x_t = x \}$ under the fixed flow of measures $\delta$ and control policy $\pi$ which specify the dynamics. The above equations are analogous to the Kolmogorov PDE in continuous time and $\delta = \Lambda(\pi)$ represents the sequence of distributions over the control and action space in the infinite-player limit when all agents use policy $\pi$ and are initially distributed on the state space according to $\mu_0$. The maps $\Lambda$ and $\Phi$ allow us to define the  equilibrium compactly.
\begin{definition}
\label{def:NE}
{\rm{A pair $(\pi,\delta) \in \M \times (\sP(\sA) \times \sP(\sX))^{T-1}$ is called}} an  MFG equilibrium {\rm{if and only if}}
$$
\pi \in \Phi(\delta)\qquad \text{and}
\qquad \delta = \Lambda(\pi).
$$ 
\end{definition}

As discussed immediately following Definition \ref{markovdef},  the restriction to Markov policies has no impact on the value function, thus an  MFG equilibrium over policies in $\M$ is also an  MFG equilibrium over policies in $\Pi$. Our approach to establish an existence theorem is to construct a set-valued operator whose fixed points correspond to  MFG equilibria, and use Kakutani's fixed point theorem to guarantee the existence of a fixed point. As we will rely on some technical arguments therein, we follow the notation and structure in the proof of \cite[Theorem 3.3]{saldi2018markov}. In the next two subsections, we construct spaces for the set-valued operator and specify its construction, establish that its fixed points correspond to  equilibria, and prove the existence of fixed points. We now proceed with the required assumptions.   

\subsection{Model Assumptions} 
This section contains the assumptions required for existence of an  MFG equilibrium. Fix a continuous moment $w:\sX \to [1,\infty)$ on the state space, which is a map for which there exists a sequence of compact sets $(H_n)_{n=1}^{\infty} \subseteq \sX$ which are increasing (in the sense that $H_n \subseteq H_{n+1}$ for every $n$), such that $\lim_{n \to \infty} \inf_{x \in \sX \setminus H_n} w(x) = \infty$, and which satisfy $w(\cdot) \geq 1 + d_\sX(\cdot,x_0)^p$ for some $x_0 \in \sX$ and some $p \geq 1$, where $d_\sX$ denotes a metric on $\sX$ compatible with its topology. 

To treat two cases simultaneously, let $v \coloneqq \1$ be the  function of $\sX$ identically equal to one when both $c$ and $\phi$ are assumed bounded, and $v \coloneqq w$ otherwise. We proceed by defining the $v$-norm of a map $g:\sX \to \R$ by
\begin{align*}
\norm{g}_v \coloneqq \sup_{x \in \sX} \frac{|g(x)|}{v(x)}.
\end{align*}
Moreover, we let $B_v(\sX)$ denote the space of all real-valued measurable functions on $\sX$ with finite $v$-norm and let $C_v(X) \subseteq B_v(\sX)$ denote the subset of continuous functions. Both of these are Banach spaces: If $v = \1$ this is an elementary result and the $v = w$ case follows form almost identical arguments as the $v = \1$ case. Finally, we define
\begin{align}
\sP_v(\sX) \coloneqq \{\mu \in \sP(\sX) : \norm{\mu}_v < \infty\} &= \{\mu \in \sP(\sX) : \int_{\sX}v(x)\mu(dx) < \infty  \}.\label{Pvdef}
\end{align}
The following are the main assumptions required for the equilibrium existence theorem.
\begin{assumption} \label{ass1}
$\text{}$		%Adhoc way of adding a line
\begin{enumerate}[itemsep=0.5pt,label = (\roman*)]
\item The maps $c$ and $\phi$ are continuous.
\item $\sA$ is compact and $\sX$ is locally compact (every point contains a compact neighborhood).
\item There exists a constant $\alpha \geq 1$ such that 
\begin{align*}
\sup_{(a,\delta^c,\delta^s) \in \sA \times \sP(\sA) \times \sP(\sX)} \int_{\sX} w(y) \rho(dy|x,a,\delta^c,\delta^s) \leq \alpha w(x) \quad \text{ for all }x \in \sX.
\end{align*}
\item The stochastic kernel $\rho$ is weakly continuous in the sense that if $(x_n,a_n,\delta_n^c,\delta_n^s) \xrightarrow{n \to \infty} (x,a,\delta^c,\delta^s)$, then $\rho(\cdot|(x_n,a_n,\delta_n^c,\delta_n^s)) \xrightarrow{n \to \infty} \rho(\cdot|(x,a,\delta_n^c,\delta_n^s))$ under the topology of weak convergence. In addition, $\int_{\sX}w(y)\rho(dy|x,a,\delta^c,\delta^s)$ is continuous in the variables $x,a,\delta^c,$ and $\delta^s$. 
\item The initial law $\mu_0$ satisfies
%\begin{align*}
$M \coloneqq \int_{\sX} v(x) \mu_0(dx) < \infty$. 
%\end{align*}
\item There exists $R \in \R$ satisfying
$$
\sup_{(a,\delta) \in \sA \times \sP(\sA) \times \sP(\sX)}c(x,a,\delta) \leq R\ v(x), \quad \text{and} \quad
\sup_{\delta \in \sP(\sA) \times \sP(\sX)}\phi(x,\delta) \leq R\ v(x),
\qquad \forall \in \sX.
$$
\end{enumerate}
\end{assumption}

\subsection{Existence of  MFG Equilibria} 
\label{finhorsec}
This section contains a proof of the following existence theorem.

\begin{theorem} 
\label{mainexistenceth}
Under Assumption \ref{ass1} there exists an MFG equilibrium for the finite-horizon model \eqref{ref_ag_discrete_dyn}. 
\end{theorem}

To prove this result (proof at the end of this section), we first construct a set-valued map whose fixed points correspond to  equilibria.  We then use  the Kakutani fixed point theorem to complete the proof. We begin with a definition.
\begin{definition}
(Value Function) Consider the non-homogeneous Markov Decision Process given by the transition kernel $\rho_t(\cdot|x,a,\delta_t)_{t=0}^{N}$ for the fixed flow of measures $\delta$ and costs as defined above. We let $V_{t}^{\delta}:\sX \to \R$ denote the value function at time $t = 1,2,\cdots,T$. In other words, we define $V_{T}^\delta \coloneqq \phi$ and iteratively define
\begin{align*}
V_{t}^{\delta}(x) \coloneqq \min_{a \in \sA } \Big(  c(x,a,\delta_t) + \int_{\sX} V_{t+1}^{\delta}(y)\rho_t(dy|x,a,\delta_t) \Big)
\end{align*}
for $t = 0,1,2,\cdots,T-1$. Furthermore, let $V^\delta \coloneqq (V_{t}^{\delta})_{t=0}^{T}$. 
\end{definition}
Recall that $M = \int_{\sX}v(x) \mu_0(dx)$) and $\sP_v(\sX)$ defined in \eqref{Pvdef}. We now define 
%for every $t = 0,1,2, \cdots,T$, 
the spaces
\begin{align*}
  &\sP_v^t(\sX) \coloneqq \{ \mu \in \sP_v(\sX) : \int_{\sX}w(x)\mu(dx) \leq \alpha^t M \}, \\
  &\sP_v^t(\sX \times \sA) \coloneqq \{\mu \in \sP(\sX \times \sA) : \mu_1 \in \sP_v^t(\sX)\},
\end{align*}
where in the second set, $\mu_1$ denotes the state marginal of $\mu$ (i.e. $\mu_1(\cdot) = \mu(\cdot \times \sA)$). Taking $R$ as in Assumption \ref{ass1}, define
\begin{align*}
L_t = R \sum_{k=0}^{T-t} \alpha^k \quad \text{ and observe that      } \quad L_t = R + \alpha L_{t+1} \quad \text{ for every }t = 0,1,2,\cdots,T.
\end{align*}
Next, define for every $t = 0,1,\cdots,T$
\begin{align*}
C_v^t(\sX) \coloneqq \{u \in C_v(\sX) : \norm{u}_v \leq L_t \}, \quad C \coloneqq \prod_{t=0}^{T}C_v^t(\sX), \quad   \Xi \coloneqq \prod_{t=0}^{T} \sP_v^t(\sX \times \sA). 
\end{align*}
The following lemma establishes regularity properties of the value function that will be required later. 
\begin{lemma}
For any $\nu \in \Xi$ we have that $V^\nu \in C$. 
\end{lemma}
\begin{proof}
First, observe that by definition $V_T^\nu = \phi \in C_v^T$. Inductively, we have
\begin{align*}
\frac{1}{v(x)}V_t^\nu(x) &= \min_{a \in \sA} \Big[  \frac{1}{v(x)}c(x,a,\nu_t) + \int_{\sX} \frac{1}{v(x)} V_{t+1}^\nu(y) \rho_t(dy|x,a,\nu_t)\Big] \\
&\leq \min_{a \in \sA} \Big[  R + \frac{1}{v(x)} \int_{\sX}  L_{t+1}v(y) \rho_t(dy|x,a,\nu_t)\Big] \\
&\leq    R + \alpha  L_{t+1} = L_t.
\end{align*}
By assumption, $\phi$ is continuous and thus inductively, and using \cite[Proposition 7.32]{bertsekas1996stochastic}, which establishes that the Bellman operator (see the following definition) preserves continuity, continuity of all the value functions follows. 
\end{proof} 
\begin{definition}
(Bellman operator) For a given $\nu \in \Xi$, the Bellman operator acting on maps $u:\sX \to \R$ is defined by setting
\begin{align*}
T_t^\nu u(x) \coloneqq \min_{a \in \sA} \Big[c(x,a,\nu_t) + \int_{\sX}u(y)\rho_t(dy|x,a,\nu_t) \Big] \quad \text{for }t=0,1,2,\cdots,N,
\end{align*}
and defining the operator $T^\nu : C \to C$ by
\begin{align*}
(T^\nu u)_t \coloneqq 
\begin{cases}
T_{t}^\nu u_{t+1} & \mbox{for }t =0,1,\cdots,T-1 \\
\phi & \text{for } t=T.
\end{cases}
\end{align*}
\end{definition}
Observe that %for $t = 0,1,2,\cdots,N$, we have
\begin{align*}
(TV^\nu)_t(x) = (T_t V_{t+1}^\nu)(x) = 
\begin{cases}
\min_{a \in \sA} \Big[c(x,a,\nu_t) + \int_{\sX}J_{t+1}^
\nu(y)\rho_t(dy|x,a,\nu_t) \Big] = V_{t}^\nu(x) & \text{for }t < T \\
\phi(x) = V_{T}^\nu(x) & \text{for }t = T.
\end{cases}
\end{align*}
Thus the value function is a fixed point of this operator. The following result establishes that in fact $T^\nu$ maps $C$ into itself, and will be required when applying Kakutani's theorem.
\begin{lemma}
Let $\nu \in \Xi$ be arbitrary. Then for all $t \geq 0$ the operator $T_t^\nu$ maps $C_v^t(\sX)$ into $C_v^{t+1}(\sX)$. 
\end{lemma}
\begin{proof}
Let $u \in C_v^t(\sX)$. Since the Bellman operator preserves continuity (see \cite[Proposition 7.32]{bertsekas1996stochastic}) we immediately have that $T_t^\nu u$ is continuous. Because
\begin{align*}
\norm{T_t^\nu u}_v = \sup_{x \in \sX} \frac{|(T_t^\nu u)(x)|}{v(x)} &= \sup_{x \in \sX} \frac{ \Big| \min_{a \in \sA} \Big[ c(x,a,\nu_{t}) +  \int_{\sX} u(y) \rho(dy|x,a,\nu_{t})  \Big] \Big|}{v(x)} \\
&\leq \sup_{(x,a) \in \sX \times \sA } \frac{c(x,a,\nu_{t}) + \int_{\sX} |u(y)| \rho(dy|x,a,\nu_{t})  }{v(x)} \\
&\leq \sup_{(x,a) \in \sX \times \sA } \frac{c(x,a,\nu_{t}) + L_t \int_{\sX} v(y) \rho(dy|x,a,\nu_{t})  }{v(x)} \\
&\leq \sup_{(x,a) \in \sX \times \sA } \frac{ R v(x) +  \alpha L_tv(x) }{v(x)} = L_{t+1},
\end{align*}
it follows that $T_t^\nu u \in C_v^{t+1}(\sX)$, as claimed.
\end{proof}
We are ready to define a set valued operator which whose fixed points will correspond to MFG equilibria. Recall that a pair $(\pi,\delta) \in \Pi \times (\sP(\sA) \times \sP(\sX))^{T+1} $ is an  MFG equilibrium  if and only if we have both $\pi \in \Phi(\delta)$ (optimality) and $\delta \in \Lambda(\pi)$ (consistency). Let $\Gamma : \Xi \to 2^{\sP(\sX \times \sA)^{T+1}}$  denote a set-valued operator defined by $\Gamma(\nu) = C(\nu) \cap B(\nu)$ where 
\begin{align*}
C(\nu) \coloneqq & \left\{\nu' \in \sP(\sX \times \sA)^{T+1} :  \nu_{t+1,1}'(\cdot) = \int_{\sX \times \sA} \rho(\cdot | x,a,\nu_{t})\nu_t(dx,da) \text{   for every }t, \quad \nu_{0,1}' = \mu_0 \right\},\\
B(\nu) \coloneqq & \Big\{ \nu' \in \sP(\sX \times \sA)^{T+1} :  \\
&\hspace{40pt} \nu_t' \Big( \Big\{ (x,a) :c(x,a,\nu_{t}) + \int_{\sX} V_{t+1}^{\nu}(y) \rho(dy|x,a,\nu_{t}) = T_t^\nu V_{t+1}^\nu(x)   \Big\} \Big) = 1 \text{ for every }t \geq 0  \Big\}.
\end{align*}

\begin{lemma}
    \label{lem:Nash}
    For any fixed point $\nu$ of $\Gamma$, there exists a 
    corresponding MFG equilibrium. 
\end{lemma}
\begin{proof}
Let $\nu \in \Xi$ be a fixed point of $\Gamma$. We then define a Markov randomized control policy by disintegrating the measure $\nu_t$ into its marginal on $\sX$ and the conditional measures, which we define to be $\pi_t$. In other words, we have
\begin{align*}
\nu_t(dx,da) = \pi_t(da|x)\nu_{t}^s(dx) \quad \text{ for every }t = 0,1,2,\cdots,T.
\end{align*}
Writing $\pi \coloneqq (\pi_t)_{t=0}^{T-1}$, it follows that $(\nu,\pi)$ constitutes an  MFG equilibrium for our setup with optimality and consistency following from the definitions of $C$ and $B$ above. In particular, optimality follows because the random control is supported on the set of maximizers of the Bellman operator (see \cite[Theorem 17.1]{hinderer1970decision} for a proof of this result). 
\end{proof}

We are now ready to complete the existence proof which uses the restrictions placed on the state-component marginals in the definition of $\Xi$.

\begin{proof}[Proof of Theorem \ref{mainexistenceth}]
To show that $\Gamma$ admits a fixed point we make use of \cite[17.55 Corollary (Kakutani–Fan–Glicksberg)]{aliprantis2006infinite}. We must check that $\Gamma$ maps a non-empty compact convex subset of a locally convex Hausdorff topological vector-space into itself, and when restricted to this subset, its graph is closed and it has non-empty convex values. Of course, our candidate subset is $\Xi$ and we take the ambient space to be the $T+1$-fold tuple of finite, signed measures on $\sX \times \sA$.  

We first show that $\Gamma(\nu) \subseteq \Xi$ for any $\nu \in \Xi$. Indeed,
it suffices to show that for an arbitrary $\nu \in \Xi$ we have that $C(\nu) \in \Xi$. As such, let $\nu' \in C(\nu)$. We need to check that $\nu_t' \in \sP_v^t(\sX \times \sA)$ for every $t \geq 0$. In other words, we need to check that $\nu_{t,1}' \in \sP_v^t(\sX)$ for all $t \geq 0$. By definition, we know that $\nu_{0,1}' = \mu_0$ thus we have that
\begin{align*}
\int_{\sX} w(x)\nu_{0}^{s'}(dx) = \int_{\sX} w(x)\mu_0(dx) = \alpha^0M,
\end{align*}
and so by definition, $\nu_{0,1}' \in \sP_v^0(\sX)$. For any other $t$, simply observe that by definition of $C(\nu)$, we have
\begin{align*}
\int_{\sX}w(y)\nu_{t+1,1}'(dy) &= \int_{\sX \times \sA} \int_{\sX} w(y) \rho(dy|x,a,\nu_{t})\nu_t(dx,da) \\
&\leq \int_{\sX \times \sA} \alpha w(x) \nu_t(dx,da) = \int_{\sX} w(x)\alpha \nu_{t}^s(dx) \leq \alpha^{t+1}M,
\end{align*}
where the penultimate inequality follows from Assumption \ref{ass1}.(iii) and the last inequality by the fact that $\nu \in \Xi$. Convexity of $\Gamma(\nu)$ for $\nu \in \Xi$ is immediate, due to convexity of each of $C(\nu)$ and $B(\nu)$. 

We omit the remaining details here, as they follow by a simplification (since we are working with a finite instead of infinite product space) of the arguments in \cite[Proposition 3.9 and Proposition 3.10]{saldi2018markov}, and we have used the notation therein. In particular, compactness and closedness of $\Xi$ rely on the use of the continuous moment function $w$, and the closedness of the graph of $\Gamma$ follow using analytic arguments.    
\end{proof}
%\hfill$\Box$

\section{Continuous-Time MFG Equilibria as Discrete-Time Limits}\label{sec:cts}
We now turn to the problem of establishing convergence of discrete-time MFG equilibria, as studied in Section \ref{sec:discrete}, to continuous-time MFG equilibria for models involving jump process dynamics of controlled intensity with interaction via the control mean. Our motivation to study this problem is two-fold: first,  no general existence result for such dynamics (to the best of our knowledge) exists in the literature, yet such dynamics have been used in concrete scenarios such as the cryptocurrency model in \cite{li2019mean}, with equilibria conjectured to exist from the convergence of numerical schemes. Secondly, the discrete-time to continuous-time convergence result provides rigorous justification for the solving of a discrete-time MFG problem via fixed point iterations as an approximate solution to the continuous time MFG, and we illustrate the effectiveness of this scheme by numerically solving the discrete-time version of \cite{li2019mean}. 

\subsection{MFG of Controlled Intensity}
Although we prove our convergence result for relatively simple dynamics in Definition \ref{agent_dyns_def} below, the result can likely be established for much more general jump-diffusion MFG dynamics. The dynamics in Definition \ref{agent_dyns_def} are motivated by the cryptocurrency model discussed in the following sections. Our method makes use of the convergence methods used for stochastic control problems considered in detail in \cite{kushner2001numerical}. We will work with a weak formulation and make use of relaxed controls and compactness arguments. 

\begin{definition}\label{relaxcont}
(Relaxed Control) Let $\sU$ denote a compact subset of $\R$. A relaxed control taking values in $\sU$ is a random measure $m \in \sP(\sU \times [0,T] )$ such that almost surely $m(\sU \times [0,t]) = t$ for every $t \in [0,T]$. The relaxed control is admissible w.r.t. a filtration $(\sF_t)_{t \in [0,T]}$ iff $t \mapsto m_t(B)$ is progressively measurable w.r.t $(\sF_t)_{t \in [0,T]}$ for any fixed $B \in \sB(\sU)$, where $m_t(da)dt = m(da,dt)$.  
\end{definition}
We make use of relaxed controls because, if $(m^{(n)})_{n=1}^{\infty}$ denotes a sequence of relaxed controls taking values in $\sP(\sU \times [0,T])$ (each possibly defined on its own probability space), then their process laws admit a weak limit. This follows because $\sP(\sU \times [0,T])$ , endowed with the topology of weak convergence of measures, is a compact Polish space. Thus $(\sL(m^{(n)}))_{n=1}^{\infty} \subseteq \sP(\sP(\sU \times [0,T]))$ is a tight family and, by Prokhorov's theorem, admits a weak limit. In contrast, it would be difficult to extract a limit from the laws of arbitrary progressively measurable control processes on $[0,T]$ taking values in $\sU$. 
\begin{definition} \label{agent_dyns_def}
(Agent Dynamics) First, fix a flow of measure on the control space $\eta = (\eta_t)_{t \in [0,T]} \subseteq \sP(\sU)$, an intensity function $\lambda: \sU^2 \to \R_{\geq 0}$, and constants $c > 0$ and $r > 0$. We say that state, control, and jump processes $X,m,N$ defined on a common filtered probability space $(\Omega,\sF,(\sF_t)_{t \in [0,T]},P)$ satisfy the dynamics considered herein iff
\begin{align} \label{state_eq_cont}
X_t = X_0 - \int_{0}^{t} \int_{\sU} c \alpha m_t(d \alpha) dt + rN_t,
\end{align}
where 
$N$ is an $\F$-adapted unit-jump process with stochastic intensity given by 
\begin{align} 
\lambda_t^{m,\eta} \coloneqq \int_{\sU}\int_{\sU} \lambda(\alpha,h)m_t(d \alpha)\eta_t(dh),
\end{align}
and where $m = (m_t)_{t \in [0,T]}$ is a relaxed admissible control.
\end{definition}
\begin{remark} \label{weak_well_posedness}
(Intensity Control Representation) Unlike the more classical controlled diffusion setup where a Brownian filtration is fixed apriori, the jump process in Definition \ref{agent_dyns_def} cannot be fixed before the control process is specified when working with intensity control models. It is therefore convenient to work under a weak formulation (as done in \cite{kushner2012weak} when establishing limit theorems for stochastic control problems), where the underlying probability space is allowed to vary with the control. We refer the reader to \cite[Chapter VII.2]{bremaud1981point} for more details on the formulation of jump-intensity control problems.  
\end{remark}

By taking $\sU = [0,L]$ for some $L > 0$ and taking  $\lambda(a,h) = a/(a + Mh)$ 
for some large constant $M > 0$ representing the number of players, we recover a relaxed version of the dynamics considered in \cite{li2019mean}, where $(\eta_t)_{t \in [0,T]}$  is interpreted as a relaxation of the mean background hash-rate of the agent
population (and not the distribution of agents' controls themselves), and $X$ denotes the reference agent wealth process. In their model, a representative miner of a cryptocurrency under the proof-of-work protocol, such as Bitcoin, hashes at rate $\alpha\geq0$ to affect the 
probability of success, namely discovery, which comes as a Poisson-type arrival with a reward $r>0$. However, hashing comes at a marginal cost (of electricity) $c>0$, hence the linear drift in \eqref{state_eq_cont}. The miner's goal is to maximize expected utility $\varphi$ of wealth $X$ at time $T$.

The following characterization \cite[T9 Theorem]{bremaud1981point} (due to Watanabe) of a stochastic intensity Poisson process will be very useful throughout.
\begin{definition}
Let $(N_t)_{t \in [0,T]}$ be a non-explosive jump process with unit jumps, adapted to a filtration $F \coloneqq (\sF_t)_{t \in [0,T]}$, and let $\lambda \coloneqq (\lambda_t)_{t \in [0,T]}$ be a $F$- progressively measurable process. Then $N$ is a doubly-stochastic Poisson (or Cox) process with stochastic intensity $\lambda$ iff 
\begin{align}
N_t - \int_{0}^{t}\lambda_s\,ds, \quad t \in [0,T]
\end{align}
is a martingale. 
\end{definition}
For a given intensity function $\lambda: \sU^2 \to \R_{\geq 0}$, let $\mathcal{T}$ denote a tuple
$(\Omega,\sF,\F,P,\eta,N,m,X)$, where 
\begin{itemize}[itemsep=0.5pt]
\item $(\Omega,\sF,P)$ is a probability space equipped with a filtration $\F \coloneqq (\sF_t)_{t \in [0,T]}$;
\item $\eta \coloneqq (\eta_t)_{t \in [0,T]} \subseteq \sP(\sU)$ is a deterministic measure flow on the action space $\sU \subseteq \R$;
\item and $m = (m_t)_{t \in [0,T]}$ is a relaxed admissible control with corresponding state and jump processes $X$ and $N$ satisfying the dynamics in Definition \ref{agent_dyns_def}. 
\end{itemize}
Given a measure $\eta_t \in \sP(\sA)$, we denote its first moment by $\overline{\eta_t} = \int h \eta_t(dh)$. 

We now give a precise definition of the weak MFG equilibria that we consider in this section.
\begin{definition}  \label{cont_mfg_def}
A relaxed MFG of controlled jump intensity for a given intensity function $\lambda: \sU^2 \to \R_{\geq 0}$, terminal reward function $\phi$ and constants $c>0$ and $r>0$ is a tuple $\mathcal{T}$ such that

\begin{itemize}[itemsep=0.5pt]
\item the control is optimal in the sense that for any other tuple $(\Omega',\sF',\F',P',\eta,N',m',X')$ satisfying the previous conditions, we have that $E^P[\phi(X_T)] \geq E^{P'}[\phi(X_T^{'})]$;%, where $\phi$ is the terminal reward function.  
\item and consistency holds, i.e. $E[\overline{m}_t] = \overline{\eta}_t$ for Lebesgue almost every $t \in [0,T]$. %\blue{What are the overlines denoting?}
\end{itemize}
If $\eta_t$ happens to be a Dirac probability measure for Lebesgue almost every $t \in [0,T]$, then we say that the MFG equilibrium is sharp. If for almost every $\omega \in \sF$ the measure $m_t(\omega)$ is a Dirac for Lebesgue almost every $t \in [0,T]$, we say that the MFG equilibria is of sharp controls. 
\end{definition}
%\Break

\subsection{Time-Discretized MFG \& Convergence Theorem}
We aim to establish an MFG existence result for the above setup as a limit of discrete-time MFG equilibria. 

We use the following natural (due to the Poisson limit theorem) discretizations approximations, parameterized by $n \in \N$. 
\begin{definition} \label{bernouli_process_def}
(Discretization Scheme) For each $n \in \N$, the discrete-time, finite horizon state process $(x_k^{(n)})_{k=0}^{T2^n}$ follows
\begin{align} \label{discrete_time_family}
x_{k+1}^{(n)} &\sim \alpha_k^{(n)} \delta_{ \{x_{k}^{(n)} - (ca_k^{(n)}/2^n) + r  \} }(\cdot) + \big(  1 - \alpha_k^{(n)} \big) \delta_{ \{x_{k}^{(n)} - (ca_k^{(n)}/2^n)  \} }(\cdot),
\end{align}
where $x_{k+1}$ is taken independent of all other random quantities up to time $k$. We allow for two schemes by taking either:
\begin{align} \label{def_alpha}
\text{Scheme 1:} \quad  \alpha_k^{(n)}= \frac{1}{2^n}\int \lambda(a_k^{(n)},dh)\eta_k^{(n)}(dh), \quad \text{Scheme 2:} \quad  
\alpha_k^{(n)} &= \frac{1}{2^n} \lambda(a_k^{(n)},\overline{\eta}_k^{(n)}).
\end{align}
It is assumed that $x_0^{(n)} \sim \mu_0$. The optimality criterion is the maximization of $E[\phi( x_{T2^n}^{(n)} ) ]$.
\end{definition}
In either scheme, $(\eta_k)_{k=0}^{2^nT} \subseteq \sP(\sU)$, which we shall refer to as the parametrizing sequence, is a fixed but arbitrary deterministic sequence, and $a_k^{(n)}$ is the $\sU$-valued control at time $k$. The convergence theorem will hold under different subsets of the following assumptions. 
\begin{assumption} 
\label{multiple_assumptions}
\begin{enumerate}[itemsep=0.5pt]
\item $\sU$ is a closed interval, and the intensity function $\lambda$ is continuous in both of its variables, and hence uniformly continuous and bounded given that its domain is compact.  \label{a1}
\item $\mu_0$ is compactly supported. \label{a2}
\item $\phi$ is %\sout{assumed strictly increasing,} 
continuous and bounded; and $\lambda(\cdot,h)$ is Lipschitz uniformly over the choice of $h \in \sU$. 
%\sout{ {\color{red}  and $\lambda(a,\cdot)$ is Lipschitz uniformly over  $a \in \sU$. NOT NEEDED} } 
\label{a3}
\item $\phi$ is strictly increasing and $\lambda(\cdot,h)$ is strictly concave for fixed $h \in \sU$. \label{a5}
\end{enumerate}
\end{assumption}
We now state and prove the main result from this section.
\begin{theorem}  \label{main_conv_th}
Suppose that \ref{a1}, \ref{a2}, and \ref{a3} from Assumption \ref{multiple_assumptions} hold. Then there exists an MFG equilibrium in the sense of Definition \ref{cont_mfg_def} for the continuous time game of controlled intensity. If additionally \ref{a5} holds, then the equilibria is of sharp control.
\end{theorem}
The proof is organized in four steps. Along the way, we prove four technical lemmas. In step 1, we interpolate the discrete-time controlled Bernoulli chains (Definition \ref{bernouli_process_def}) to obtain continuous-time processes, establish that the sequence of joint state/control laws of the interpolations admits a weak limit, and extract a limiting flow of control measures. In step 2, we establish that the limiting law corresponds to a controlled jump process with the dynamics specified in Definition \ref{cont_mfg_def}. In step 3, consistency with respect to the flow of measures obtained from step 2 is established. In step 4, we establish optimality. We then conclude with a brief note regarding the last assertion of the theorem. 

\textbf{Step 1 (Discrete-Time Approximations):} For simplicity and without loss of generality, let the finite time horizon $T$ of the continuous-time model be an integer. Consider the MFG arising from the discretization introduced in Definition \ref{bernouli_process_def} under Scheme 2, where we use $y^{(n)}$ to denote the state process. Applying Theorem \ref{mainexistenceth}, we obtain an MFG equilibrium $((a_k^{(n)})_{k=0}^{T2^n},(\zeta_k^{(n)})_{k=0}^{T2^n})$, where the $\zeta$'s are elements of $\sP(\sU)$ satisfying $\sL(a_k^{(n)})= \zeta_k^{(n)}$ for every $k$,
and where $(a_k^{(n)})_{k=0}^{T2^n}$ is the optimal control process for the fixed sequence $(\zeta_k^{(n)})_{k=0}^{T2^n}$ with respect to the maximization of $E[\phi( y_{T2^n}^{(n)} ) ]$. 

%\Break

We now define new discretizations under Definition \ref{bernouli_process_def} Scheme 1 with state process denoted by $x^{(n)}$ and with parametrizing sequence $\eta_k^{(n)} \coloneqq \delta_{\{\overline{\zeta}_k^{(n)}\}},k=0,1,\cdots,2^nT-1$. Here, $x^{(n)}$ will be the discrete-time dynamics to be interpolated.  We will slightly abuse notation by denoting the continuous-time control and background flows with the same letters as in discrete-time, but will differentiate using time index $t$ and $k$ when referring to the continuous-time and discrete-time processes, respectively. 
For each $n \in \N$, we construct continuous-time processes $(X_t^{(n)})_{t \in [0,T]}$, $(a_t^{(n)})_{t \in [0,T]}$, and $(\eta_t^{(n)})_{t \in [0,T]}$ as follows. For each $t \in [0,T]$: 
\begin{align*}
X_t^{(n)} \coloneqq  x_k^{(n)} - ca_k^{(n)}\Big(t - \frac{k}{2^n}\Big), \quad \eta_t^{(n)} \coloneqq \eta_k^{(n)}, \quad   \text{and} \quad a_t^{(n)} \coloneqq  a_k^{(n)} \quad \text{ for} \quad t \in \Big[\frac{k}{2^n},\frac{k+1}{2^n}\Big).  
\end{align*}
In words, $X^{(n)}$ linearly interpolates the drift, with a jump at multiples of $1/2^n$ whenever the discrete-time process jumps up. By construction, sample paths of $X^{(n)}$ are right-continuous with left limits for every $n \in \N$. By defining $m_t^{(n)} \coloneqq \delta_{\{a_t^{(n)}\}}$ we obtain a relaxed control representation of the control process $a_t^{(n)}$. The following representation (which defines $N_t^{(n)}$) will be very useful:
\begin{align} \label{ncontproc}
X_t^{(n)} = X_0^{(n)} - c \int_{0}^{t}a_s^{(n)}ds + rN_t^{(n)} = x_0^{(n)} - c \int_{0}^{t} \int_{\sU}a m_s^{(n)}(da) ds + rN_t^{(n)},
\end{align}
where $N_t^{(n)}$ is a jump process taking values in $\{0,1,2,3,\cdots\}$, with unit jumps possible only on dyadic rationals of order $n$. Note that the random variables used to define the above continuous processes are a countable family. As such, we assume them to be defined on a common probability space $(\Omega,\sF,P)$. 

Next, we extract limit points from the joint process laws $\sL(X^{(n)},m^{(n)},N^{(n)})$ as $n \to \infty$. Note that the processes $X^{(n)}$ naturally take values in the space $D[0,T]$ (right-continuous functions with left limits on $[0,T]$ taking values in $\R$) which can be endowed with the Skorokhod topology. Together with this topology, the space $D[0,T]$ is a Polish space. For a collection of probability measures on the Borel sets of a Polish space, tightness is equivalent to sequential compactness (i.e. every sequence of measures from the collection admits a further weakly convergent sub-sequence) as per Prokhorov's theorem. Establishing tightness will be made easy due to \cite[Theorem 9.2.1]{kushner2001numerical}, which is restated in the Appendix for convenience. Because the drift term is linear and the jump probability is scaled with $n$, it is almost immediate that the sequences of laws $\sL(X^{(n)})_{n \in \N}$ and $\sL(N^{(n)})_{n \in \N}$ satisfy the referenced tightness condition and are thus tight. Tightness of the laws $\sL(m^{(n)})_{n=1}^{\infty}$ is immediate (see the discussion following Definition \ref{relaxcont} of relaxed controls) since we have assumed a compact action space $\sU$. 

As we have checked tightness of each of the sequence of marginals, we conclude that $$\sL(X^{(n)},m^{(n)},N^{(n)})_{n=1}^{\infty}$$ is tight and admits a weakly converging sub-sequence. Similarly, we can extract a weak limit from the sequence $(\eta^{(n)})_{n=1}^{\infty}$ (since these are measures on the compact space $[0,T] \times \sU$), and thus we can  always find a sub-sequence along which both sequences converge. For simplicity of notation, we will dispense with this sub-sequence and assume that the convergence is as $n \to \infty$. 

By the Skorokhod representation theorem (see \cite[Theorem 9.1.7]{kushner2001numerical}), there exists some probability space $(\Omega,\sF,P)$ supporting random variables $(\tilde{X}^{(n)},\tilde{m}^{(n)},\tilde{N}^{(n)})_{n=1}^{\infty}$ converging almost surely to $(X,m,N)$ and such that $\sL(\tilde{X}^{(n)},\tilde{m}^{(n)},\tilde{N}^{(n)}) = \sL(X^{(n)},m^{(n)},N^{(n)})$. It is immediate that \\$(\tilde{X}^{(n)},\tilde{m}^{(n)},\tilde{N}^{(n)})$ satisfies the representation \eqref{ncontproc}, and hereafter we abuse notation by dropping tilde. %from the notation. 

To summarize, we have that for $P$-almost every $\omega$, $X^{(n)}(\omega) \xrightarrow{n \to \infty} X(\omega)$ w.r.t the Skorokhod topology on $D[0,T]$, $m^{(n)}(\omega) \xrightarrow{n \to \infty} m(\omega)$ in the weak topology, and redundantly, $N^{(n)}(\omega) \xrightarrow{n \to \infty} N$ in the Skorokhod topology on $D[0,T]$. The $\eta^{(n)}$'s are deterministic and $\eta^{(n)} \xrightarrow{\sL, n \to \infty} \eta$. 

We define:
\begin{align*}
\sF_t &\coloneqq \sigma(X_s,m_s,N_s : s \leq t) \quad \text{ for every } t \in [0,T], \\
\sF_t^{(n)} &\coloneqq \sigma(X_s^{(n)},m_s^{(n)},N_s^{(n)} : s \leq t) \quad \text{ for every } t \in [0,T], n \in \N.
\end{align*}
It is important to note that we can always define a derivative $m_t(\omega)$ such that $m_t(\omega)(A)$ is $(\sF_t)_{t \in [0,T]}$ adapted for each $t \in [0,T]$ and $A \in \sB([0,L])$, and such that the disintegration $m(\omega)(da,dt) = m_t(\omega)(da)dt$ holds \cite[Section 9.5]{kushner2001numerical}. Similarly we can also do this for approximations, with the Borel measurability of maps of the form $t \mapsto E[\overline{m}_t] = \int \alpha m_t(\omega)(d\alpha) dP(\omega)$ following from progressive measurability (see for example \cite[Theorem 3.1]{florescu2012young}).

\textbf{Step 2 (Characterizing the Limit Point):}
Using the almost sure convergence, it follows that 
$(X,m,N)$ satisfy
\begin{align}
X_t(\omega) = \lim_{n \to \infty}  X_t^{(n)}(\omega) &= \lim_{n \to \infty} \Big( X_0^{(n)}(\omega) - c \int_{0}^{t}\int_{\sU} a m_s^{(n)}(\omega)(da) ds + rN_t^{(n)}(\omega) \Big) \\
&=  X_0(\omega) - c \int_{0}^{t}\int_{\sU}a m_s(\omega)(da) ds + rN_t(\omega). 
\end{align}
Clearly, $N$ is the law of a unit jump process. The fact that it is in fact an $\F$-jump process follows by arguments as in \cite[Equation 10.1.8]{kushner2001numerical}. We verify that it has the correct (stochastic) intensity by fixing $0 < s < t$ and denoting
$$ J_{s,t} = \int_{s}^{t} \int_{\sU}\int_{\sU} \lambda(a,h)m_\rho(da) \eta_{\rho}(dh) d\rho. $$
Then we have
\begin{align*}
\E\left[N_t - J_{0,t} \Big| \sF_s \right] &=  \E\left[N_t - N_s - J_{s,t} \Big| \sF_s \right]
+ N_s - J_{0,s} \\
&= \lim_n E\Big[N_t^{(n)} - N_s^{(n)} - \int_{s}^{t} \int_{\sU}\int_{\sU}\lambda(a,h)m_\rho^{(n)}(da)\eta_{\rho}^{(n)}(dh)d\rho \Big| \sF_s \Big]+ N_s - J_{0,s}\\
&=  N_s - J_{0,s},
\end{align*}
where the limit can be seen to equal zero from the construction of the approximations. 

\textbf{Step 3 (Consistency):} We state consistency as a Lemma. 
\begin{lemma}
(Consistency of the limit) For Lebesgue-almost every $t \in [0,T]$ we have that $E[\overline{m}_t] = \overline{\eta}_t$.
\end{lemma}
\begin{proof}
As discussed earlier, Borel measurability of the map $t \mapsto E[\overline{m}_t]$ follows from admissibility of the control process $m$. Observe that for any Borel set $B \subseteq [0,T]$ we can find a sequence $C([0,T];\R) \ni f_l \xrightarrow{l \to \infty} \1_B$ in $L^1$ with $\norm{f_l}_\infty \leq 1$ for every $l$, from which, using the triangle inequality, we see that
\begin{align*}
\Big| \int_B  E[\overline{m}_t] dt - \int_B \overline{\eta}_t dt \Big| &\leq \Big| \int \1_B(t) E[\overline{m}_t] dt - \int f_l(t) E[\overline{m}_t] dt \Big| \\
&+ \Big| \int f_l(t) E[\overline{m}_t] dt - \int f_l(t) E[\overline{m}_t^{(n)}] dt \Big| \\   
&+ \Big| \int f_l(t) E[\overline{m}_t^{(n)}] dt - \int \int f_l(t) a \eta_t^{(n)}(da) dt \Big| \\ 
&+ \Big| \int \int f_l(t) a \eta_t^{(n)}(da) dt   - \int \int f_l(t) a \eta_t^{}(da) dt \Big| \\ 
&+ \Big| \int \int f_l(t) a \eta_t(da) dt   - \int \int \1_B(t) a \eta_t(da) dt \Big| \\
& \xrightarrow{n,l \to \infty} 0,
\end{align*}
where terms one and five converge due to the $L^1$ approximation, term three is identically zero by consistency of the discrete-time chains, and term four converges to zero by definition of weak convergence. For the second term, we rewrite as
\begin{align*}
\Big|  \int f_l(t) \int_{\Omega} \int a m_t(\omega)(da) dP(\omega) dt - \int f_l(t) \int_{\Omega} \int a m_t^{(n)}(\omega)(da) dP(\omega) dt \Big| \\
= \Big| \int_{\Omega} \Big[ \int \int f_l(t)    a m_t(\omega)(da)dt -  \int \int f_l(t)    a m_t^{(n)}(\omega)(da) dt \Big] dP(\omega) \Big| \\
= \Big| \int_{\Omega} \Big[ \int f_l(t) a m(\omega)(da,dt) -  \int f_l(t)    a m^{(n)}(\omega)(da,dt) \Big] dP(\omega) \Big|,
\end{align*}
from which we see that the integrand converges to zero for almost every $\omega$ since $m_t^{(n)}(\omega) \xrightarrow{n \to \infty} m_t(\omega)$ for $P$-almost every $\omega$ in the topology of weak convergence. An application of the Dominated Convergence Theorem yields that term two also converges to zero. Consistency therefore follows.   
\end{proof}

\textbf{Step 4 (Establishing Optimality):}
For a fixed discretization integer $n$ and sequence of measures $\eta^{(n)} \coloneqq (\eta^{(n)})_{k=0}^{T2^n - 1}$, the expected cost of an adapted control chain $a \coloneqq (a_k)_{k=0}^{2^nT}$ and the value function of the discrete-time system (dynamics given by Scheme 1 in Definition \ref{bernouli_process_def} with $\eta^{(n)}$  the parameterizing sequence) are denoted by
\begin{align*}
w^{(n)}(k,x,a,\eta^{(n)}) = E[\phi(x_{2^nT}^{(n),a} )|x_k^{(n),a} = x], \quad v^{(n)}(k,x,\eta^{(n)}) = \sup_{a}w^{(n)}(k,x,a,\eta^{(n)}),
\end{align*}
where the superscript $a$ on the state emphasizes that $a$ is the driving control. Analogously, we use capital letters for the equivalent continuous time quantities
\begin{align*}
W(t,x,m,\eta) = E[\phi(X_{T}^m)|X_t = x], \quad V(t,x,\eta) = \sup_{m}W(t,x,m,\eta),
\end{align*}
where this time $m \coloneqq (m_t)_{t=0}^{T}$ denotes a relaxed admissible (possibly random) control process and the underlying probability space may vary with $m$. The dynamics of $X$ are of course constrained to be a jump process of controlled jump intensity from Definition \ref{cont_mfg_def}. The following lemma (which follows the construction for controlled diffusion in \cite[Theorem 3.5.2]{kushner2012weak}) establishes that allowing for relaxed controls does not improve the value function. (We will of course be interested in the map $t \mapsto \lambda(t,\eta_t)$ for a fixed flow of measure $(\eta_t)_{t \in [0,T]}$, where the objective is the maximization of expectation of continuous terminal wealth utility).

\begin{lemma} \label{chatlemma}
(Chattering Lemma) Consider the continuous-time intensity control problem from Definition \ref{cont_mfg_def} with compact action space $\sU$ and time-varying intensity function $f_t:\sU \to \R_{\geq 0}$, assumed continuous for each $t \in [0,T]$.  For any $\gamma > 0$, there exists a finite set $\sU^\gamma \coloneqq \{\alpha_1^\gamma,\cdots,\alpha_k^\gamma\} \subseteq \sU$ such that, for any admissible relaxed control $m$ and associated jump and state processes $N$ and $X$ all defined on some probability space $(\Omega,\sF,P)$, there exists a piecewise constant control $u_m^\gamma$, with discontinuities only at dyadic rationals of a fixed order (and possibly defined on another probability space), such that
\begin{align}
\big| W(t,x,m) - W(t,x,u_m^\gamma) \big| < \gamma.
\end{align}
\end{lemma}
\begin{proof}
We first prove it without the dyadic rational requirement. First, fix $\rho > 0$ and let \\ $B_1^\rho,B_2^\rho,\cdots,B_{N_\rho}^\rho$ denote a disjoint partition of $\sU$ with the diameter of each of the sets not exceeding $\rho$. For a stochastic admissible control $m$, a fixed $\Delta > 0$, and an integer $i$, define the following (random) numbers:
\begin{align*}
r_{i,j}^{\rho,\Delta}(\omega) \coloneqq \int_{i\Delta}^{(i+1)\Delta}m_s(\omega)(B_j^\rho)ds = m(\omega)(B_j^\rho \times [i\Delta, (i+1) \Delta]) \quad \text{for} \quad j=1,2,\cdots,N_\rho.
\end{align*}
Observe that, by construction, we have
%\begin{align}
$\Delta = \sum_{j=1}^{N_\rho}r_{i,j}^{\rho,\Delta}$. 
%\end{align}
Next, subdivide the next interval (to preserve causality) $[(i+1)\Delta,(i+2)\Delta)$ into intervals of length given by the $r_{i,j}^{\rho,\Delta}$'s. On the interval of length $r_{i,j}^{\rho,\Delta}$ define $u_m^{\gamma,\rho}$ to take on the value $\alpha_j^\rho$. Note that here, the $r$'s are dependent on $\omega$ and so the resulting ordinary control is itself stochastic. Let $m^{\rho, \Delta}$ denote the relaxed representation of this control. Note that the piecewise constant control is defined on the original probability space $(\Omega,\sF,P)$. By possibly augmenting the probability space or modifying the probability measure $P$ (recall Remark \ref{weak_well_posedness}), let $N^{\rho,\Delta}$ denote another jump process with stochastic intensity given by $\lambda_t^{\rho,\Delta} = \int_{\sU}f_t(\alpha)m_t^{\rho,\Delta}(d \alpha)$ and let $X^{\rho, \Delta}$ and $N^{\rho, \Delta}$ satisfy 
\begin{align}
X_t^{\rho, \Delta} = X_0 - \int_{0}^{t}c \alpha m_s^{\rho,\Delta}(d\alpha)ds + rN_t^{\rho,\Delta}. 
\end{align}
By construction, the laws $\{\sL(m^{\rho,\Delta},N^{\rho,\Delta}) : \rho > 0,\Delta > 0 \}$ are tight and we let $\tilde{m}$ and $\tilde{N}$ denote a limit point. Clearly, $\sL(\tilde{m}) = \sL(m)$ by construction, thus it follows that the limit point of the above is in fact the original law $\sL(m,N)$. From here, we simply observe that for any fixed $\rho$ and $\Delta$, we can define an integer $z(\rho,\Delta)$ and round the jump times of the control up to the nearest $z(\rho,\Delta)$'th order dyadic rational (i.e. an element in $0,1/2^{z(\rho,\Delta)},2/2^{z(\rho,\Delta)},3/2^{z(\rho,\Delta)},\cdots$). Taking $z(\rho,\Delta) \xrightarrow{\rho,\Delta \to 0} \infty$ sufficiently fast, the arguments above remain true.  
\end{proof}
It follows from the above that there is no gain in expanding the class of controls in terms of improving the value function. We prove another useful lemma:

%\Break

\begin{lemma}
Recall that $V(t,x,\eta)$ denotes the value function for the system with a fixed measure flow $\eta_{t \in [0,T]}$. If $\eta^{(n)} \xrightarrow{\sL, n \to \infty} \eta$, then we have that $V(\cdot,\cdot,\eta^{(n)}) \xrightarrow{n \to \infty} V(\cdot,\cdot,\eta)$ point-wise. 
\end{lemma}
\begin{proof}
Let $(a_t)_{t \in [0,T]}$ be a control process defined on some probability space $(\Omega,\sF,P)$ which supports two jump processes {with} stochastic rates $\int \lambda(a_t,h)\eta_t(dh)$ and $\int \lambda(a_t,h)\eta_t^{(n)}(dh)$. Fix $\omega \in \Omega$ and let $C([0,T];[0,L]) \ni f_k \xrightarrow{L^1,k \to \infty} a(\omega)$. For any $t \in [0,T]$
\begin{align}
\Big| \int_{t}^{T} \int &\lambda(a_s(\omega),h)\eta_s(dh) ds - \int_{t}^{T} \int \lambda(a_s(\omega),h)\eta_s^{(n)}(dh) ds \Big| \\
&\leq \Big| \int_{t}^{T} \int \lambda(a_s(\omega),h)\eta_s(dh) ds - \int_{t}^{T} \int \lambda(f_k(s),h)\eta_s(dh) ds \Big|\\ 
&+ \Big| \int_{t}^{T} \int \lambda(f_k(s),h)\eta_s(dh) ds - \int_{t}^{T} \int \lambda(f_k(s),h)\eta_s^{(n)}(dh) ds \Big| \\
&+ \Big| \int_{t}^{T} \int \lambda(f_k(s),h)\eta_s^{(n)}(dh) ds - \int_{t}^{T} \int \lambda(a_s(\omega),h)\eta_s^{(n)}(dh) ds \Big|,
\end{align}
with terms one and three going to zero using the uniform Lipschitz condition on $\lambda(\cdot,h)$ and the $L^1$ convergence, and the second term going to zero by definition of weak convergence. Because the intensities converge to one another almost surely, so does the expected terminal reward, establishing the convergence of the claim. 
\end{proof}

We now connect the value functions of discrete time (considering Scheme 1 from Definition \ref{bernouli_process_def}) and continuous time systems. When using a flow $(\eta_t)_{t \in [0,T]}$ of control measure to specify transition dynamics of the discrete time systems, we discretize the measures by averaging, as follows:
\begin{align} \label{meas_disc}
\eta(n)_k(B) = 2^n \int_{k/2^n}^{(k+1)/2^n} \eta_t(B) dt \text{        for       } k = 0,1,\cdots,2^nT-1, \quad \quad \eta(0)^{(n)} = \eta_0.
\end{align}

\begin{lemma} \label{aproxlem_correct}
Let $(a_t)_{t=0}^{T}$ denote a control process which is piece-wise constant, right-continuous with left limits, and jumps only on dyadic rationals of order $n'$. Then letting $a^{(n)}$ denote its discretization (i.e. sampling) on dyadic rationals of order $n$ we have that 
\begin{align} \label{aprox_eq1}
\Big| w^{(n)}(a^{(n)},\eta) - w^{(n)}(a^{(n)},\eta^{(n)}) \Big| \xrightarrow{n \to \infty} 0,
\end{align}
and
\begin{align} \label{aprox_eq2}
|w^{(n)}(a^{(n)},\eta) - W(a,\eta)| \xrightarrow{n \to \infty} 0.
\end{align} 
\end{lemma}
\begin{proof}
For the first limit, we complete the proof by bounding the difference between the expected rewards earner corresponding to the a continuous-time interval of the form $[l/2^{n'},(l+1)/2^{n'})$. Note that for a given $\omega$, $a_{l2^{n-n'}}^{(n)}(\omega)$ is constant in $n$. For $n > n'$ we have
\begin{align*}
&\Big| \sum_{k=0}^{2^{n-n'}-1} \frac{1}{2^n} \int_{\sU}  \lambda(a_{l2^{n-n'}}^{(n)},h)\eta(n)_{l2^{n-n'} + k}(dh)  - \sum_{k=0}^{2^{n-n'}-1} \frac{1}{2^n} \int_{\sU}  \lambda(a_{l2^{n-n'}}^{(n)},h)\eta^{(n)}(n)_{l2^{n-n'} + k}(dh)  \Big| \\
=& \Big| \sum_{k=0}^{2^{n-n'} - 1}  \int_{l/2^{n'} + k/2^n}^{l/2^{n'} +(k+1)/2^n} \int  \lambda(a_{l2^{n-n'}}^{(n)},h)\eta_{t}(dh)dt  - \sum_{k=0}^{2^{n - n'}-1}  \int_{l/2^{n'} + k/2^n}^{l/2^{n'} + (k+1)/2^n} \int_{\sU} \lambda \Big(a_{l2^{n-n'}}^{(n)}, h  \Big)  \eta^{(n)}_t(dh) dt \Big| \\
= &\Big| \int_{l/2^{n'}}^{(l+1)/2^{n'}} \int  \lambda(a_{l2^{n-n'}}^{(n)},h)\eta_{t}^{(n)}(dh)dt  -  \int_{l/2^{n'}}^{(l+1)/2^{n'}} \int_{\sU} \lambda \Big(a_{l2^{n-n'}}^{(n)}, h  \Big)  \eta_t(dh) dt \Big| \xrightarrow{n \to \infty} 0,
\end{align*}
from which the result follows. The second limit follows immediately from the Poisson limit theorem.  
\end{proof}

\begin{proof}[Proof of Theorem \ref{main_conv_th}]
We can finally conclude optimality. Recall that $\eta$ and $m$ denote the limit flow of measure resp. control obtained by taking a weak limit of the interpolated discrete-time approximations. Suppose now for a contradiction that $m$ is not optimal for $\eta$. Then using the Chattering Lemma \ref{chatlemma} there exists a piece-wise constant control $\tilde{a}$ (with jumps on finite dyadics no finer than of order $n'$) such that $W(m,\eta) < W(\tilde{a},\eta)$. We sample $\tilde{a}$ for each $n$ at $n$-th order dyadics to obtain a discrete control process $\tilde{a}^{(n)}$. Combining the two limits established in Lemma \ref{aproxlem_correct}, we have that
\begin{align}
\lim_{n \to \infty} w^{(n)}(\tilde{a}^{(n)},\eta^{(n)}) = W(\tilde{a},\eta) > W(m,\eta) = \lim_{n \to \infty} w^{(n)}(a^{(n)},\eta^{(n)}),
\end{align}
which for all $n$ sufficiently large, contradicts the optimality of $a^{(n)}$ for the fixed sequence of measures $\eta^{(n)}$ (note that the discretization scheme from \eqref{meas_disc} applied to $(\eta_t^{(n)})_{t \in [0,T]}$ recovers the original sequence $(\eta^{(n)}_k)_{k=0}^{2^nT-1}$ that was interpolated to obtain $(\eta_t^{(n)})_{t \in [0,T]}$). As such, we conclude that $m$ is indeed optimal for $\eta$. 

\Break

We have established the first assertion of the theorem. For the second assertion regarding sharpness of the limit control $m$, assume $\phi$ is strictly increasing and that $\lambda(\cdot,h)$ is strictly concave for each $h$. Let $(m,\eta)$ denote a relaxed MFG equilibrium. Fix $t \in [0,T]$, define $a^*_t = \int a m_t(da)$ where $dm = dm_tdt$, and note that 
\begin{align}
\int_{\sU} \lambda(a_t^*,h)\eta_t(dh) \geq \int_{\sU} \int_{\sU} \lambda(a,h)m_t(da)\eta_t(dh).
\end{align}
Thus, letting $X^{a^*}$ and $X^{m^*}$ denote the wealth processes (possibly defined on distinct probability spaces) driven by the controls $\delta_{\{a_t^*\}}$ and $m$ respectively, we have that $E[\phi(X_T^{a^*}) ] \geq E[\phi(X_T^{m}) ]$, with equality if and only if $m_t$ is a Dirac mass for Lebesgue almost every $t \in [0,T]$ almost surely. By optimality of $m$, it follows immediately that $m$ is itself (up to redefining on Lebesgue null sets) a sharp control. This concludes the proof. 

\end{proof}

\section{Cryptocurrency Mining MFG}\label{sec:crypto}
We now return to the cryptocurrency mining MFG model from \cite{li2019mean}. We provide MFG existence results for both the original continuous-time model and a discrete-time analog. For the latter, we numerically illustrate that damped fixed-point iterations of the discrete-time Bellman and Kolmogorov equations converge to an MFG equilibrium, and we recover the qualitative preferential attachment behavior obtained from the continuous-time model in \cite{li2019mean} under CRRA wealth utility. We note that the numerical scheme utilized in this paper is distinct from the fixed point iterations used to solve the continuous-time MFG in \cite{li2019mean}. There, finite difference methods are used to solve the HJB and Kolmogorov PDEs, whereas we exactly solve a space-discretized discrete-time game, and we need not scale the dynamics so that time steps correspond to small intervals of time.

\subsection{Discrete-Time Model}
We begin by considering a family of discrete-time models parametrized by $n \in \N$ and $\epsilon \geq 0$. Using Scheme 2 from Definition \ref{bernouli_process_def}, we define
\begin{align} \label{regularized_intensity}
\lambda^{(\epsilon)}(a,h) \coloneqq \frac{a}{a + hM + \epsilon} \1_{\{a > 0\}},
\end{align}
for a constant $M > 0$ (which proxies for the number of miners, see \cite{li2019mean} for details). Given a fixed sequence of probability measures $(\zeta_k^{(n,\epsilon)})_{k=0}^{2^nT - 1} \subseteq \sP(\sA)$, the transition dynamics for the wealth process $(x_k^{(n,\epsilon)})_{k=0}^{2^nT}$ are given by
\begin{align} \label{discrete_time_family2}
x_{k+1}^{(n,\epsilon)} \sim \frac{\lambda^{(\epsilon)}(a_k^{(n,\epsilon)},\overline{\zeta}_k^{(n,\epsilon)})}{2^n}  \delta_{ \{x_{k}^{(n,\epsilon)} - (ca_k^{(n,\epsilon)}/2^n) + r  \} } + \Big(  1 -  \frac{\lambda^{(\epsilon)}(a_k^{(n,\epsilon)},\overline{\zeta}_k^{(n,\epsilon)})}{2^n}  \Big)   \delta_{ \{x_{k}^{(n,\epsilon)} - (ca_k^{(n,\epsilon)}/2^n)  \} },
\end{align}
for $k = 0,1,\cdots,T2^n - 1$. The control objective is the maximization of expected terminal utility of wealth, $E[\phi(x_{T2^n}^{(n,\epsilon)})]$. Assuming that the initial wealth law $\mu_0$ is compactly supported, that the action space is a compact interval $[0,L]$, and that the wealth utility $\phi$ is bounded (which is WLOG since agents can increase their wealth by at most $2^nTr$ over the course of the game), it is straightforward to check that, for $\epsilon > 0$, Theorem \ref{mainexistenceth} applies, guaranteeing the existence of an MFG equilibrium for the game arising from the above dynamics. The condition $\epsilon > 0$ is required in order to ensure continuity of the transition kernel. For the remainder of this section, we fix $n \in \N$ and will extract an MFG equilibrium as a limit of equilibria by taking $\epsilon \to 0$.   

Fix $n$ and $\epsilon > 0$, and let $\nu^{(n,\epsilon)} \coloneqq (\nu_k^{(n,\epsilon)})_{k=0}^{2^nT}$ denote a sequence of laws on $\sX \times \sA$ which characterize an MFG equilibrium (recall that the control policy is obtained via the disintegration as in Section 2) for the game arising from the dynamics in \eqref{discrete_time_family2}. For ease of notation, let $\mu_k^{(n,\epsilon)} = \nu_{k,1}^{(n,\epsilon)} $ and $\eta_k^{(n,\epsilon)} = \nu_{k,2}^{(n,\epsilon)}$ where $\nu_{k,1}^{(n,\epsilon)}$ and $\nu_{k,2}^{(n,\epsilon)}$ denote the state and control marginals of $\nu_k^{(n,\epsilon)}$, respectively. In this specific setup, $\sX = \R$ and $\sA = [0,L]$. Recall also that the optimal control 
$$a_k^{(n,\epsilon)} \sim \pi_k^{(n,\epsilon)}(\cdot|x_k^{(n,\epsilon)})\quad \mbox{where}\quad \nu_k^{(n,\epsilon)}(dx,da) = \pi_k^{(n,\epsilon)}(da|x)\mu_k^{(n,\epsilon)}(dx).$$ 
We now extract a weak limit $\nu^{(n,\epsilon)} \to \nu^{(n)}$ as $\epsilon \to 0$ and will establish that $\nu^{(n)}$ is an MFG equilibrium for the cryptocurrency mining model \eqref{discrete_time_family2} with $\epsilon = 0$. Since $n \in \N$ is fixed we suppress it from the superscript notation hereafter, and also write $\lambda = \lambda^{(0)}$ for simplicity. 

We begin by observing that if 
\begin{align}
B \coloneqq \{k \in \{0,1,\cdots,2^nT - 1 \}: \eta_k = \delta_{\{0\}} \}
\end{align}
is non-empty, then an optimal control for the original ($\epsilon = 0$) model with fixed background hash-rate $(\overline{\eta}_k)_{k \in \{0,1,2,\cdots,2^nT\} }$ will in general not exist: This follows because any non-zero hash rate will result in a unit jump intensity for the reward process, whereas zero hash-rate results in zero reward process intensity, hence zero hash-rate on $B$ is sub-optimal for any reasonable choice of terminal utility/initial condition. On the other hand, any control process that is non-zero on $B$ can be improved by making it even closer to zero while keeping it positive. Fortunately, assuming the following mild assumption, we can show that the limit hash-rate is indeed positive for every time $k$. 

\begin{assumption} \label{mildassumption}
Assume that $\phi$ is non-decreasing and that the initial wealth distribution is not supported on the set of maximizers for the function $\phi$, i.e.
\begin{align}
\mu_0 \{ \argmax_{y \in \R} \phi(y) \} < 1.
\end{align}
\end{assumption}

\begin{proposition} \label{tech_lemm_d}
Under Assumption \ref{mildassumption}, the limit (as $\epsilon \to 0$ with $n$ fixed but arbitrary) sequence of measures on the control space $(\eta_k)_{k=0}^{2^nT} = (\nu_{k,2})_{k=0}^{2^nT}  \subseteq \sP([0,L])$ is never a Dirac at zero. In fact, there exists some $d > 0$ such that $\overline{\eta}_k > d$ for every $k = 0,1,\cdots, 2^nT - 1$ and we can find such a $d$ that holds for any choice of $n \in \N$. 
\end{proposition}
\begin{proof}
Suppose for a contradiction that there is some $k$ such that $\overline{\eta}_k = 0$. It follows that $\lim_{\epsilon \to 0} \overline{\eta}_k^{(\epsilon)} = 0$. Recall that $a^{(\epsilon)}$ is an optimal control for the fixed background hash-rate $\eta^{(\epsilon)}$. For a given $\epsilon$, let $\nu(\epsilon) \coloneqq \epsilon \lor \overline{\eta}_k^{(\epsilon)} $, and define on the same probability space a new control process given by
\begin{align}
\tilde{a}_j^{(\epsilon)} =
\begin{cases}
a_j^{(\epsilon)} & j \neq k  \\
\sqrt{\nu(\epsilon)} \lor a_j^{(\epsilon)} & j = k,
\end{cases}
\end{align}
for $j = 0,1,\cdots,2^nT-1$. We will complete the proof by showing that for sufficiently small $\epsilon$, $\tilde{a}^{(\epsilon)}$ results in higher expected wealth utility than $a^{(\epsilon)}$. We only need to compare the controls at time $k$. The expected rate of the Bernouli reward at time $k$ of the original control satisfies
\begin{align}
E \Big[ \frac{1}{2^n} \lambda^{(\epsilon)}(a_k^{(\epsilon)},\overline{\eta}_k^{(\epsilon)}) \Big] \leq \frac{1}{2^n}  \lambda^{(\epsilon)}(E[a_k^{(\epsilon)}],\overline{\eta}_k^{(\epsilon)})dt = \frac{\overline{\eta}_k^{(\epsilon)}}{2^n((M+1) \overline{\eta}_k^{(\epsilon)} + \epsilon)} \leq \frac{1}{2^n(M+1)}.
\end{align}
where we have used MFG consistency in the equality. On the other hand, we also have that
\begin{align}
E \Big[ \frac{1}{2^n} \lambda^{(\epsilon)}(\tilde{a}_k^{(\epsilon)},\overline{\eta}_k^{(\epsilon)}) \Big] \geq  \frac{\sqrt{\nu(\epsilon)}}{2^n(\sqrt{\nu(\epsilon)} + \overline{\eta}_k^{(\epsilon)}M + \epsilon)} \geq \frac{\sqrt{\nu(\epsilon)}}{2^n(\sqrt{\nu(\epsilon)} + \nu(\epsilon)(M + 1))}  \xrightarrow{\epsilon \to 0} 2^{-n}. 
\end{align}
Observe that the difference in costs between the two controls converge as $\epsilon \to 0$. Also, Assumption \ref{mildassumption} implies that uniformly over $\epsilon \geq 0$, the choice of control processes (since the intensity is bounded), and the possible values of $k$, there is some positive probability that the wealth $x_k^{(n,\epsilon)}$ has not reached $ \argmax_{y \in \R} \phi(y)$. As such, taking $\epsilon \to 0$ results in a converging of the expected costs of the controls $\tilde{a}^{\epsilon}$ and $a^{\epsilon}$, whereas the number of rewards of the former is strictly higher than the latter, with the difference not converging as $\epsilon \to 0$. As such, the former control will eventually result in strictly higher expected terminal wealth utility compared to the latter for $\epsilon$ sufficiently small. This contradicts optimality of $a^\epsilon$ for the fixed sequence $\eta^\epsilon$, completing the proof. 

To prove the last statement regarding the uniformity of $d$ over the choice of $n$, we make the dependency on $n$ explicit. Note that the positive probability that the wealth $x_k^{(n,\epsilon)}$ has not reached the supremum of the terminal wealth utility is itself uniform in $n$; this follows using the Poisson limit theorem. Also, the dependency of $n$ on the bound 
\begin{align}
\liminf_{\epsilon \to 0} \Big( E \Big[ \frac{1}{2^n} \lambda^{(\epsilon)}(\tilde{a}_k^{(n,\epsilon)},\overline{\eta}_k^{(n,\epsilon)}) \Big] - E \Big[ \frac{1}{2^n} \lambda^{(\epsilon)}(a_k^{(n,\epsilon)},\overline{\eta}_k^{(n,\epsilon)}) \Big] \Big) \geq \frac{M}{2^n(M+1)}
\end{align}
is offset by the fact that as $n$ increases, the difference in cost of the controls $a_k^{(n,\epsilon)}$ and $\tilde{a}_k^{(n,\epsilon)}$ converges to zero faster (also with a $1/2^n$ factor) as $\epsilon \to 0$.
\end{proof}

We are now ready to establish existence of an MFG equilibrium as the limit of equilibria for the model parametrized by $\epsilon$ for a fixed $n \in \N$. 
\begin{proposition}
Recall that ($n \in \N$ is fixed and suppressed from the notation) $\nu^{(\epsilon)}$ denotes (the laws of) an MFG equilibrium for the model parametrized by $\epsilon > 0$ and $\nu$ is a weak limit as $\epsilon \to 0$. This object constitutes an MFG equilibrium for the model with $\epsilon = 0$. 
\end{proposition}
\begin{proof}
We have shown that none of the measures in the sequence $(\eta_k)_{k=0}^{2^nT}$ are Dirac measures at zero. Let $a_k \sim \pi_k(\cdot|x_k)$ where $\nu_k(da,dx) = \pi_k(da|x)\nu_{k,1}(dx)$. Using the continuity and boundedness of the map
\begin{align}
 (\sX \times \sA)^{2^nT + 1} \ni (x_k,a_k)_{k=0}^{2^nT} \mapsto \phi(x_{2^nT}  ) \in \R,
\end{align}
and the weak convergence $\nu^{(\epsilon)} \xrightarrow{\epsilon \to 0} \nu$, it follows that $E[\phi(x_{2^nT}^{(\epsilon)})] \xrightarrow{\epsilon \to 0} E[\phi(x_{2^nT})]$. Suppose now for a contradiction that there exists a control policy $\tilde{\pi} = (\tilde{\pi}_k)_{k=0}^{2^nT - 1}$ which outperforms $a$ under the fixed background sequence $(\eta_k)_{k=0}^{2^nT - 1}$. Let $\tilde{x}^{(\epsilon,\eta^{(\epsilon)})}$ and $\tilde{a}^{(\epsilon,\eta^{(\epsilon)})}$ denote (each possibly defined on its own probability space) the resulting control and state processes under the policy $\tilde{\pi}$ using the reward probability function $\lambda^{(\epsilon)}$ and the fixed sequence of measures $\eta^{(\epsilon)}$. Using the fact that $\overline{\eta}_k^{(\epsilon)} \xrightarrow{\epsilon \to 0} \overline{\eta}_k > 0$ for every $k = 0,\cdots,2^nT$, one obtains that $E[\phi(\tilde{x}^{(\epsilon,\eta^{(\epsilon)})}_{2^nT})] \xrightarrow{\epsilon \to 0} E[\phi(\tilde{x}^{(0,\eta)}_{2^nT})] > E[\phi(x_{2^nT})]$ which, for all $\epsilon > 0$ sufficiently small, contradicts the optimality of the control policy obtained by disintegration of $\nu^{(\epsilon)}$ for the fixed background flow $\eta^{(\epsilon)}$ and probability reward function $\lambda^{(\epsilon)}$. Note that consistency of the limit measures follows because it holds for every $\epsilon > 0$. 
\end{proof}
To summarize, we have shown that for any $n \in \N$ there exists an MFG equilibrium for the discrete-time game with $\epsilon = 0$. Moreover, the population control distributions are bounded away form zero, uniformly over the time step and the parameter $n \in \N$. Before moving to continuous-time, we solve the discrete-time game numerically. 

\subsection{Numerical Computation of Discrete-Time MFG} 

We solve the discrete-time MFG for parameters $n = 1$, $T = 300$, $M = 1000$ (players), take the terminal utility to be a CRRA utility $\phi(x) = 2 x^{1/2}$, and assume normally distributed initial wealth. The population hash-rate is initizlied to be constant in time, and the following steps are then iterated until convergence of the population hash-rate is observed:
\begin{enumerate}[noitemsep]
\item Compute the optimal control policy using dynamic programming, given the fixed background hash-rate from the previous step.
\item Compute the resulting hash-rate using the initial state law and the optimal control policy from step 1.
\item Update the new population hash-rate as a convex combination of that obtained in the last two steps. 
\end{enumerate}
Step 3 is a damping step which is required to attain convergence. The damping factor can be tuned to find a trade-off between convergence and speed. Solving with a damping factor of 0.9, we observe convergence from any choice of initial constant population hash-rate, and obtain numerical wealth distribution dynamics and optimal control policies for the reference agent (plotted below). 

The results obtained are qualitatively similar to those obtained from the numerical PDE approach in \cite{li2019mean}, where we see that agents drop out of the game if their wealth falls below a (time-dependent) threshold and observe the preferential attachment phenomenon where a small percentage of the population becomes increasingly wealthy, with the majority of miners eventually dropping out of the game with comparatively little wealth. Figures \ref{fig:enter-label} and \ref{fig:enter-label2} illustrate the discrete-time convergence results. Convergence speed is highly dependent on the choice of initial condition, in part due to the use of a large damping factor. The stability of the algorithm suggests uniqueness of the equilibria.

\begin{figure}[htbp]
    \centering
    \includegraphics[width=0.9\linewidth]{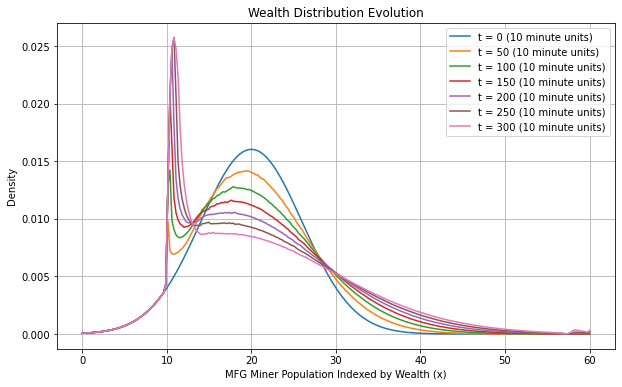}
    \caption{Evolution of Wealth Distribution}
    \label{fig:enter-label}
\end{figure}
\begin{figure}[htb]
    \centering
    \includegraphics[width=0.9\linewidth]{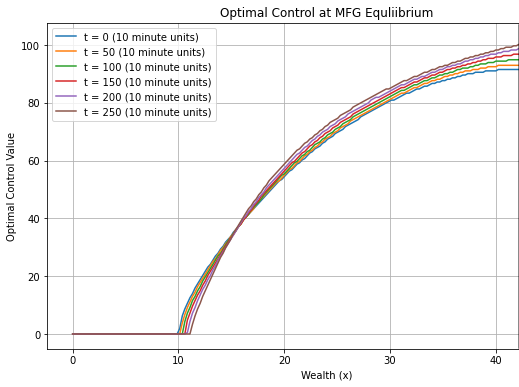}
    \caption{Optimal Control at Equilibrium}
    \label{fig:enter-label2}
\end{figure}

\subsection{Continuous Time Existence}

To apply the convergence result Theorem \ref{main_conv_th} (and Theorem \ref{mainexistenceth} for the discrete-time games), we again work with the intensity map $\lambda^{(\epsilon)}$ defined in \eqref{regularized_intensity}. Taking $\sU = [0,L]$ and assuming the terminal wealth utility is continuous and bounded and that the initial state law is compact, it is straightforward to check that for $\epsilon > 0$, Theorem \ref{main_conv_th} applies, guaranteeing the existence of a relaxed MFG equilibrium (of {\em a priori} relaxed controls) for the continuous-time game. As in the discrete-time case, we turn to the problem of extracting a limit for the original model by taking $\epsilon \to 0$. First, we note the following:
\begin{proposition}
If the terminal wealth utility $\phi$ is strictly increasing, any MFG equilibrium for the continuous-time cryptocurrency model is of sharp controls. 
\end{proposition}
%\clearpage
\begin{proof}
Follows from the second assertion of Theorem \ref{main_conv_th}.  
\end{proof}
Suppose now that for a given $\epsilon > 0$, $(X^{(\epsilon)},N^{(\epsilon)},m^{(\epsilon)},\eta^{(\epsilon)})$ is an MFG equilibrium with the state, jump, and control processes defined on a common probability space (which may vary with $\epsilon$). As in discrete-time, our aim is to extract a weak limit of process laws by taking $\epsilon \to 0$. Let $(X,N,m)$ denote processes (defined on a possibly distinct probability space) such that 
\begin{align} \label{final_mfg_eq}
\sL(X^{(\epsilon)},N^{(\epsilon)},m^{(\epsilon)}) \xrightarrow{\sL,\epsilon \to 0} \sL(X,N,m). 
\end{align}
We need to establish that $(X,N,m,\eta)$ is an MFG equilibrium for the cryptocurrency mining model with $\epsilon = 0$. Similarly to the discrete-time case, if the set $B \coloneqq \{t \in [0,T]: \eta_t = \eta_{\{0\}} \}$
has positive Lebesgue measure, then an optimal control for the $\epsilon = 0$ model with fixed control measure flow $(\eta_t)_{t \in [0,T]}$ will in general not exist. Under Assumption \ref{mildassumption}, we will see that $B$ has zero Lebesgue measure.

\begin{proposition}
The limit disintegration $(\eta_t)_{t \in [0,T] } \subseteq \sP([0,L])$ is not a Dirac at zero except for possibly on a Lebesgue null set.
\end{proposition}
\begin{proof}
Recall that the object $\eta$ has been obtained by, for each $\epsilon$, interpolating and then embedding $\overline{\eta}^{(n,\epsilon)} \in [0,L]^{n+1}$ into a space of measures, and taking $n \to \infty$ to obtain $\eta^{(\epsilon)} \in \sP([0,T] \times [0,L])$. Then, we take $\epsilon \to 0$ to obtain $\eta$. By Proposition \ref{tech_lemm_d}, however, we know that there is some $d > 0$ such that $\overline{\eta}_t^\epsilon > 0$ for every $t \in [0,T]$ and every $\epsilon$ sufficiently small, from which the result follows. 
\end{proof}
We conclude with the following proposition. 
\begin{proposition}
The tuple $(X,N,m,\eta)$ defined above \eqref{final_mfg_eq} constitutes an  MFG equilibrium for the cryptocurrency mining model with $\epsilon = 0$. The equilibrium is relaxed in the sense of Definition \ref{cont_mfg_def}, but is of sharp control.
\end{proposition}
\begin{proof}
We have shown that the limit hash-rate $\overline{\eta}_t \geq d$ for every $t \in [0,T]$. We now show that the control $m$ is optimal for the fixed measure flow $\eta$. Note the continuity (note that continuity on $D([0,T])$ is with respect to the Skorokhod topology for which the evaluation map at the endpoint $T$ is continuous), and boundedness of the map 
\begin{align} \label{cont_b_map}
D([0,T]) \times \sP([0,T] \times [0,L]) \ni (N,m) \mapsto \phi \Big(rN_T - c \int_{0}^{T} a m(dt,da) \Big). 
\end{align}

Suppose now for a contradiction that there exists a measure-valued control process $\tilde{m}$ which outperforms $m$ under the fixed background flow $\eta$ and intensity function $\lambda$. Then we have that (expectations taken on possibly different probability spaces)
\begin{align} \label{contrd}
E[\phi(rN_T - c \int_{0}^{T} m_tdt )] < E[\phi(r\tilde{N}_T^{(0)} - c \int_{0}^{T} \overline{\tilde{m}}_tdt )] = \lim_{\epsilon \to 0} E[\phi(r\tilde{N}_T^{(\epsilon)} - c \int_{0}^{T} \overline{\tilde{m}}_tdt )],
\end{align}
where $\tilde{N}^{(\epsilon)}$ denotes a unit jump process with stochastic intensity $\int \int \lambda^{(\epsilon)}(a,h) \tilde{m}_t(da) \eta_t^{(\epsilon)}(dh)$. The above limit is justified via the continuity and boundedness of the map in \eqref{cont_b_map} combined with the fact the fact that $\eta_t$ is never a Dirac at zero. For $\epsilon > 0$ sufficiently small, this contradicts the optimality of $m^{(\epsilon)}$ for the hash-rate $\eta^{(\epsilon)}$. As such, we conclude that $m$ is optimal for the hash-rate $\eta$ as desired. Consistency follows by continuity, taking limits, and using the fact that consistency holds for every $\epsilon > 0$, and we thus conclude that the tuple $(X,N,m,\eta)$ constitutes an  MFG equilibrium for the cryptocurrency mining model with $\epsilon = 0$. The equilibrium is relaxed in the sense of Definition \ref{cont_mfg_def}, but is of sharp control. 
\end{proof}
%\begin{remark}
\subsection{Uniqueness and Sharpness of MFG Equilibrium}\label{sec:uniq}
So far, we have not discussed the question of MFG uniqueness for the cryptocurrency MFG model. Numerically, one observes that, for reasonable wealth utilities, the same MFG solution is obtained from fixed-point iterations independently of starting conditions of the algorithm, suggesting uniqueness in these cases. Consider the discrete-time reference agent problem for a given $n \in \N$ with fixed population hash-rates $(\overline{\eta}_k^{(n)})_{k=0}^{2^nT}$ and assume an increasing and strictly concave utility function $v_{2^nT}^{(n)} = \phi$. Economic intuition suggests that strict concavity and monotonicity extend to the value functions at all times. Under strict concavity of the value functions, one can show that
\begin{align}
a_k^{(n)*}(x) = \argmax_{a \in [0,L]}\{E[v_{k+1}^{(n)}(x_{k+1}) | x_k^{(n)} = x, a_k^{(n)} = a ] \}  
\end{align}
is a singleton. This follows by observing that the objective function in the above maximization is strictly quasi-concave. Moreover, using  \cite[Theorem 1]{levin2003supermodular} one can see that the quantity $a_k^{(n)*}(x)$ is strictly decreasing in $\overline{\eta}_k^{(n)}$ everywhere except possibly on a small neighborhood of zero. Provided that one can show that any MFG equilibrium results in a population hash-rate that is not in this small neighborhood (which may follow in certain cases from arguments as in Proposition \ref{tech_lemm_d}), then uniqueness of equilibrium follows from a simple contradiction argument by assuming two distinct equilibria. Note also that uniqueness of equilibrium for each $n \in \N$ allows one to conclude that the continuous-time equilibrium hash rate $(\eta_t)_{t \in [0,T]}$ is in fact $[0,L]-$valued (and not $\sP([0,L])$ valued), that is, sharpness of MFG equilibrium. 
%\end{remark}

Because the question of establishing conditions for uniqueness and sharpness of the MFG equilibrium are specific to the wealth utility and parameter choices, we leave this for future work.

\section{Conclusion}
In this paper, we have accomplished three tasks. First, a general discrete-time MFG existence theorem was established, involving general transition dynamics with mean-field interactions via both the states and controls, and influencing both the transition dynamics and costs. Second, the discrete-time result was used to obtain relaxed MFG equilibria existence results for models of controlled jump intensity with mean-field interaction via the controls, and affecting the intensity of the jump processes. Finally, the results were applied to provide existence guarantees for a cryptocurrency mining MFG model, and an alternative numerical scheme, motivated by the discrete-time to continuous-time convergence result was implemented. This scheme was shown to coincide (in the sense of obtaining similar qualitative agent behaviour) with numerical solutions to the original continuous-time cryptocurrency mining MFG which was solved by numerically solving coupled Kolmogorov and HJB PDEs. 

\appendix
\section{Tightness Theorem}
The following result is restated here for convenience, and its proof can be found in \cite[Theorem 9.2.1]{kushner2001numerical}.
\begin{theorem} \label{tcritkushner}
(Tightness Criteria for the space $D$) Consider an arbitrary collection (possibly uncountable) of processes $\{X^{(\alpha)} : \alpha \in I\}$ taking values in the space $D^k[0,\infty)$ ( $\R^k$-valued cadlag functions on $[0,\infty)$) and defined on a common probability space $(\Omega,\sF,P)$. Assume that for each $\delta>0$ and rational $t \in [0,\infty) \cap \Q$ there exists a corresponding compact set $K_{\delta,t}$ such that
\begin{align*}
\sup_{\alpha \in I}P(X_t^{(\alpha)} \notin K_{\delta,t}) \leq \delta.
\end{align*}
Let now $\sF_t^{(\alpha)} \coloneqq \sigma\{X_s^{(\alpha)} : s \leq t\}$ and let $\sT_T^{(\alpha)}$ denote the set of $\sF_t^{(\alpha)}$ stopping times that are bounded by $T$. Suppose now that for each $T \in [0,\infty)$ we have that 
\begin{align*}
\lim_{\delta \to 0}\sup_{\alpha \in I} \sup_{\tau \in \sT_T^{(\alpha)}} E(\1 \land |X_{\tau + \delta}^{(\alpha)} - X_{\tau}^{(\alpha)}| ) = 0.
\end{align*}
Then the family of laws $\sL(X^{(\alpha)})_{\alpha \in I}$ is tight. 
\end{theorem}

\bibliography{references} 

\end{document}